\documentclass[letterpaper,10pt,conference]{ieeeconf}

\IEEEoverridecommandlockouts
\usepackage{amsmath,amssymb,amsthm,amsfonts}
\usepackage{mathtools}
\mathtoolsset{showonlyrefs}
\usepackage{graphicx}
\graphicspath{{figures/}}
\usepackage[caption=false]{subfig}
\let\labelindent\relax
\usepackage{enumitem}

\newcommand{\di}{\mathrm{d}}
\newcommand{\mc}{\mathcal}
\newcommand{\R}{\mathbb R}
\newcommand{\umax}{u_\mathrm{max}}

\newtheorem{remark}{Remark}
\newtheorem{definition}{Definition}
\newtheorem{theorem}{Theorem}
\newtheorem{problem}{Problem}

\begin{document}

\title{\LARGE \bf
	Boundary Control for Stability and Invariance of Traffic Flow Dynamics: A Convex Optimization Approach
}

\author{
    Maria Teresa Chiri, Roberto Guglielmi, and Gennaro Notomista
    \thanks{Maria Teresa Chiri is with the department of Mathematics and Statistics, Queen's University, Kingston, ON, Canada. {\tt\small maria.chiri@queensu.ca}}%
    \thanks{Roberto Guglielmi is with the department of Applied Mathematics, University of Waterloo, ON, Canada {\tt\small roberto.guglielmi@uwaterloo.ca}}%
    \thanks{Gennaro Notomista is with the Department of Electrical and Computer Engineering, University of Waterloo, Waterloo, ON, Canada. {\tt gennaro.notomista@uwaterloo.ca}}%
}

\pagestyle{empty}
\maketitle
\thispagestyle{empty}

\begin{abstract}
In this letter we propose an optimization-based boundary controller for traffic flow dynamics capable of achieving both stability and invariance conditions. The approach is based on the definition of Boundary Control Barrier Functionals, from which sets of invariance-preserving boundary controllers are derived. In combination 
with sets of stabilizing controllers, we reformulate the problem
as a convex optimization program solved at each point in time to synthesize the boundary control inputs. We derive sufficient conditions for the existence of optimal controllers that ensure both stability and invariance.
\end{abstract}

\section{Introduction}

Applications of control of Partial Differential Equations (PDEs) arise in several different fields, including fluid, thermal, structural, and delay systems. Typical control objectives in such scenarios include stability and invariance. The former encodes long-term objectives such as \textit{eventually} steering the system to a desired state (e.g. achieving a desired temperature profile of a metal plate), while the latter is concerned with preventing the system from exiting a desired \textit{safe} set (e.g. limiting the maximum stress in an Euler-Bernoulli beam below the yield point).

Considerable attention has been given to the synthesis of stabilizing controllers for PDEs, based on control Lyapunov functionals and backstepping with applications to traffic, thermal, and multi-agent control.
Yet, the practical implementation of such controllers requires resorting to numerical approximation and ad-hoc computational techniques \cite{vazquez2024backstepping}. Invariance has been considered in the context of optimization-based controllers and is typically encoded as state constraints to be satisfied over the optimization horizon \cite{daudin2023optimal}. However, many existing algorithms are too computationally demanding for being used in real time applications.

In this letter, we propose an optimization-based boundary control strategy for traffic flow dynamics that jointly ensures stability and invariance. Our approach draws inspiration from recent advances in stabilizing and safety-preserving controllers for ordinary differential equations, where control Lyapunov \cite{ledyaev1999lyapunov} and barrier functions \cite{ames2019control} are unified through quadratic programming to enforce stability and invariance. We introduce Boundary Control Barrier Functionals (BCBFal) for a traffic flow model and derive boundary constraints that guarantee invariance-like conditions, formulated as the positivity of the zero superlevel set of continuously differentiable functionals.
These constraints are then enforced in a convex optimization program solvable in polynomial time.

Similar approaches have been developed in \cite{hu2025boundary} and \cite{wang2023safe}. These works use barrier functions in the context of PDEs: the former leverages neural operators for PDE control \cite{lifourier}, the latter focuses on hyperbolic PDE-ODE cascades. Differently from these approaches, our proposed optimization-based control synthesis framework does not need to learn a representation of input-output mappings and is able to directly enforce constraints on any differentiable function of the state of the PDE. As a matter of fact, we utilize the whole PDE dynamics in order to define the sets of boundary controllers that ensure stability and invariance.

To summarize, the contributions of our work are the following: (i) We introduce the notion of boundary control barrier functionals for PDEs; (ii) We derive constraints on boundary control inputs in order for the state to satisfy stability and invariance conditions; (iii) We formulate a convex optimization boundary control policy able to achieve stability and invariance conditions simultaneously; (iv) We provide sufficient conditions for the existence of such optimal boundary controllers.

\section{Traffic Modeling and Problem Formulation}

In this section, we introduce the LWR model, as well as the Lyapunov and barrier control functionals that will be used for stabilization and invariance, respectively.

\subsection{Boundary value problem for the LWR Model}

The Lighthill-Whitham-Richards (LWR) model \cite{LW,R} is a fundamental macroscopic model for traffic flow which describes the conservation of the total number of vehicles and assumes that the average traffic speed \( v(x,t) \) depends solely on the traffic density \( u(x,t) \). Consequently, the mean traffic flow (i.e., the number of vehicles crossing point \( x \) per unit time) is given by  $f(x,t) = u(x, t) v(u(x, t))$, leading to the hyperbolic conservation law  $u_t + (u v(u))_x = 0$. In particular, the velocity function attains its maximum when the vehicles density is zero and decreases to zero as the density increases up to a value $u_{\max}$. Specifically, we will assume that 
\begin{equation} \label{eq:flux}
f:[0,\umax]\to\R_+,\quad f(u)=u\Big(1-\frac{u}{\umax}\Big)
\end{equation}
and denote with $\hat{u}=\umax / 2$ its unique point of maximum.

The initial-boundary value problem (IBVP) for the LWR model on the interval $(a,b)$ with initial data $u_0$ is given by
\begin{align}
   & u_t+f(u)_x=0,\qquad u(0,x)=u_0(x),\label{CL}\\
   & u(t,a)=\omega_a(t),\qquad u(t,b)=\omega_b(t).\label{BVs}
\end{align}
In the remainder of the letter, we will consider weak entropy solutions for the IBVP with BV data, i.e., weak solutions on the interval $[a,b]$ that satisfy Lax admissibility conditions and the weak boundary conditions introduced in \cite{bardos1979first}.

\subsection{Lyapunov-based Stabilization}

The boundary stabilization problem for the IBVP \eqref{CL}-\eqref{BVs} can be
formulated as follows:
\begin{problem}
Given a stationary solution $u^*$ to the conservation law in \eqref{CL}, do there exist boundary conditions $\omega_a$ and $\omega_b$ such that the IBVP is well posed and its solution is stable in the sense of Lyapunov at $u^*$?
\label{prob:stability}
\end{problem}

To address this problem, we can follow \cite{7509658,bayen2022control} where the authors consider the Lyapunov functional
\begin{equation}
    V(u(t)) = \frac{1}{2} \int_a^b (u(t,x) - u^*)^2 \di x
    \label{eq:clf}
\end{equation}
with $u$ weak entropy solution to the conservation law in~\eqref{CL}. Continuity of the mapping $t\to u(t,\cdot)$ from $[0,T]$ to $L^1(a,b)$ with $T>0$ ensures that the functional $V$ is well-defined and continuous. Since the solution to~\eqref{CL}-\eqref{BVs} with initial data in BV can be approximated by an IBVP whose solution remains piecewise smooth at all times~\cite{7509658}, we assume that $u$ is piecewise smooth. We index the jump discontinuities of $u(t)$ at time $t$ in increasing order of their location by $i=0,\dots, N(t)$, including the boundaries $a, b$ with $x_0(t)=a$ and $x_N(t)=b$ and write: 
\begin{equation}
    V(u(t))= \frac{1}{2}\sum_{i=0}^{N(t)-1}\int_{x_i(t)}^{x_{i+1}(t)} (u(t,x) - u^*)^2\,\di x.
\end{equation}
Then differentiating with respect to $t$ we get 
\begin{equation}
    \begin{aligned}
        \frac{dV}{dt}(t)=& \big(u(t,a) - u^*\big)f(u(t,a))-\big(u(t,b) - u^*\big)f(u(t,b))\\
        &-F\big(u(t,a)\big)+F\big(u(t,b)\big)\\
        & +\sum_{i=1}^{N(t)-1} \Delta_i\big(\big(u - u^*\big)f(u)-F(u)\big)\\
        &  -\sum_{i=1}^{N(t)-1} \frac{\big(u(t,x_i^-) - u^*\big)+\big(u(t,x_i^+) - u^*\big)}{2}
    \end{aligned}
    \label{eq:Vdot}
\end{equation}
where $F$ is a primitive of the flux $f$, and $\Delta_i$ is the jump at the discontinuity $x_i(t)$ given by the Rankine-Hugoniot condition.
The last two terms of \eqref{eq:Vdot} correspond to jump discontinuities in the solution. While they are neither observable nor controllable from the boundaries, they contribute to the decrease of the Lyapunov functional~\cite[Proposition 1]{7509658}.
Therefore, we aim to design stabilizing boundary controls such that
\begin{gather}\label{eq:LyapunovCond}
    \big(u(t,a) - u^*\big)f(u(t,a))-\big(u(t,b) - u^*\big)f(u(t,b))\\
    -F\big(u(t,a)\big)+F\big(u(t,b)\big)\leq{-\alpha(V(u(t)))},
\end{gather}
for a class $\mc K$ function $\alpha$. Previous control strategies proposed in the literature are greedy, non-local, and formulated in terms of the boundary traces of the solutions and the initial data~\cite{7509658}. 
The objective of this letter is to determine boundary controls as solutions to convex optimization problems able to achieve both stabilization and invariance. After introducing the latter in the next subsection, we will show how these notions can be used to recast typical control objectives for conservation laws in traffic modeling.

\subsection{Barrier-based Invariance}

The invariance problem for~\eqref{CL}-\eqref{BVs} reads as follows:
\begin{problem}
Given a desired upper bound, $0\le\bar u \le \umax$, of the solution to
~\eqref{CL}, do there exist boundary conditions $\omega_a$ and $\omega_b$ such that the IBVP is well-posed and its solution $u$ remains smaller than $\bar u$ (in a suitable norm)?
\label{prob:invariance}
\end{problem}
Similarly to how we addressed the stabilization in Problem~\ref{prob:stability} in terms of Lyapunov functionals, our approach to solve Problem~\ref{prob:invariance} is based on the novel concept of boundary control barrier functionals.
\begin{definition}[Boundary Control Barrier Functional]
    Let ${\mc S \subset L^2(a,b)}$ be the zero superlevel set of the functional $B : L^2(a,b) \to \R$. $B$ is a Boundary Control Barrier Functional (BCBFal) for the system \eqref{CL}--\eqref{BVs} if there exists a class $\mc K$ function $\beta$ such that 
    \begin{equation}
        \sup_{\omega_a, \omega_b\in[0,\umax]} \left\{-\left(\frac{\partial B}{\partial u} f_x\right)(u,\omega_a, \omega_b) + \beta(B(u)) \right\} \ge 0
    \end{equation} 
    for all $u\in\mc S$.
    \label{def:bcbf}
\end{definition}

\begin{theorem}
Let B be a BCBFal, $\mc S$ its zero superlevel set, and define the set $\mc O_u := \{ (s,z) : -\left(\frac{\partial B}{\partial u} f_x\right)(u,s, z) + \beta(B(u)) \ge 0 \}$. Any BV boundary controllers $(\omega_a, \omega_b) \in \mc O_u$ renders the set $\mc S$ forward invariant, i.e.,
if the initial condition $u(0)\in\mc S$, then $u(t)\in\mc S$ for all $t\ge 0$.
\label{thm:bcbf}
\end{theorem}
\begin{proof}
    It follows from the comparison lemma applied to the ordinary differential equation $\dot{B} + \beta(B) = 0$.
\end{proof}

\begin{remark}
Despite the similarity in the name, the BCBFal introduced in Definition~\ref{def:bcbf} differs from the neural BCBF introduced in \cite{hu2025boundary} in two main aspects. First, in our definition the value of the functional depends on the entire state, rather than on its trace only. Second, we do not need to leverage learned representations to enforce the invariance constraints encoded by the BCBFal.
\end{remark}

A natural way to tackle Problem~\ref{prob:invariance} is through the definition of the following BCBFal:
\begin{equation}
    B(u(t)) = \bar u^2 - \int_a^b u(t,x)^2 \di x,
    \label{eq:bcbf}
\end{equation}
whose zero superlevel set corresponds to the states $u$ with $L^2$-norm less than a threshold $\bar u$. In other words, we bound the total density of vehicles at each time $t$.

\section{Convex Optimization Control Policies}

For the sake of computational efficiency, we are interested in reformulating Problems 1 and 2 as convex optimization problems. To do so, we identify intervals where $\dot V$ and $-\dot{B}$ are convex with respect to $u(t,a)$ and $u(t,b)$.
It is straightforward to verify that:
\begin{itemize}
    \item   $\dot{V}$ convex in $u(t,a)$ on $\mc C_a := [0,\frac{2u^* + \umax}{4}]$
    \item $\dot{V}$ convex in $u(t,b)$ on $\mc C_b := [\frac{2u^* + \umax}{4},\umax]$
    \item $-\dot{B}$ convex in $u(t,a)$ on $\mc I_a := [0,\frac{\umax}{4}]$
    \item $-\dot{B}$ convex in $u(t,b)$ on $\mc I_b := [\frac{\umax}{4},\umax]$
\end{itemize}
For ease of notation, we introduce the following functions:
\begin{equation}
\begin{aligned}
        g(s,z) :\!&= (s - u^*)f(s)-(z - u^*)f(z) -F(s)+F(z),\\
        k(s,z) :\!&= s f(s) - z f(z) -F(s)+F(z),\\
        C(t) &= \alpha\left(V(u(t)) \right)\ge 0,\quad \alpha\text{ of class }\mathcal{K},\\
        D(t) &= \beta\left( B(u(t)) \right)\ge 0,\quad \beta\text{ of class }\mathcal{K}.
    \end{aligned}
\end{equation}
Then, Problem~\ref{prob:stability} is solved through the following convex optimization problems, whose inequality constraints encode condition~\eqref{eq:LyapunovCond}:
\begin{equation}
\min_{\omega_a(t)\in\mc C_a} \left\{ \omega_a(t)^2 ~:~ g(\omega_a(t),u(t,b)) \le -C(t) \right\}
\label{eq:optprogstabilityleft}
\end{equation}
\begin{equation}
\min_{\omega_b(t)\in\mc C_b} \left\{ \omega_b(t)^2 ~:~ g(u(t,a),\omega_b(t))\le -C(t)\right\}
    \label{eq:optprogstabilityright}
\end{equation}
\begin{equation}
\!\!\!\!\!\min_{\tiny\ \begin{aligned}&\omega_a(t)\in\mc C_a\\ &\omega_b(t)\in\mc C_b\end{aligned}} \left\{ \omega_a(t)^2+\omega_b(t)^2 : g(\omega_a(t),\omega_b(t))\le -C(t)\right\}    \label{eq:optprogstabilityleftright}
\end{equation}
where stability is sought after via the left boundary, the right one, or both, respectively.

Similarly, to tackle Problem~\ref{prob:invariance}, the boundary controls are evaluated as the solutions of one of the following convex optimization problems:
\begin{equation}
\min_{\omega_a(t)\in\mc I_a} \left\{ \omega_a(t)^2 ~:~ k(\omega_a(t),u(t,b)) \le D(t) \right\}
    \label{eq:optprogsafetyleft}
\end{equation}
\begin{equation}
\min_{\omega_b(t)\in\mc I_b} \left\{ \omega_b(t)^2 ~:~ k(u(t,a),\omega_b(t))\le D(t)\right\}
    \label{eq:optprogsafetyright}
\end{equation}
\begin{equation}
\min_{ \tiny\begin{aligned}&\omega_a(t)\in\mc I_a\\ &\omega_b(t)\in\mc I_b\end{aligned}} \left\{ \omega_a(t)^2 + \omega_b(t)^2 : k(\omega_a(t),\omega_b(t))\le D(t)\right\}
     \label{eq:optprogsafetyleftright}
\end{equation}
that seek for the optimal control that ensures invariance via the left boundary, the right one, or both, respectively.
\begin{theorem} There exist values of the parameters $u^*$, $\bar{u}\le \umax$ and class $\mc K$ functions $\alpha, \beta$ such that the optimization problems~\eqref{eq:optprogstabilityleft}, \eqref{eq:optprogstabilityright}, \eqref{eq:optprogstabilityleftright}, \eqref{eq:optprogsafetyleft}, \eqref{eq:optprogsafetyright},~\eqref{eq:optprogsafetyleftright} are feasible.
\end{theorem}
The proof is constructive and is illustrated below. We also highlight that the specific choice of the $\mathcal{K}$-class function $\alpha$ or the value of ${u}^*$ may, in some cases, prevent the existence of a stabilizing control. To address this issue, one can dynamically adjust $\alpha$ or switch between~\eqref{eq:optprogstabilityleft} and \eqref{eq:optprogstabilityright}. Similar considerations also apply to the invariance problem.

It is worth noticing that we can deal with compound problems too, where one boundary control aims to stabilize the system, and the other boundary control ensures that the invariance condition is preserved. This will be addressed in Theorem~\ref{thm:mixedpb} in Section~\ref{sec:mixedpb}.

\subsection{Stabilizing Controllers}
\label{sec:stabcontrol}

We first discuss conditions to ensure the existence of solutions to
\begin{equation}
    \min_{\omega_a(t)\in\mc C_a} \left\{ \omega_a(t)^2 ~:~ g(\omega_a(t),u(t,b)) \le -C(t) \right\}.
\end{equation}
This is a constrained minimization problem; therefore, the optimal control is either the trivial one, $\omega_a(t)=0$, provided that it lies in the interior of the region identified by the constraint, or it sits on the boundary of the region defined by the constraint, 
i.e. $\omega_a(t)$ is a root of the nonlinear equation 
\begin{equation}\label{bvl}
g(\omega_a(t),u(t,b)) = -C(t). 
\end{equation}
Let $p(s):= g(s,u(t,b))$, $\delta=\min\lbrace u^*, \hat{u}\rbrace$, $\gamma=\max\lbrace u^*, \hat{u}\rbrace$.\\ 
{\bf Case} $u^*\neq \hat{u}$: then $p'(s) < 0$ on $[0,\delta)\cup(\gamma, \umax]$ and $p'(s) > 0$ on $(\delta,\gamma)$. Since $\frac{2 u^* + \umax}{4} = \frac{u^* + \hat{u}}{2}$, equation~\eqref{bvl} has at most two roots on $\mc C_a$. In particular, we have three different cases:
 \begin{itemize} 
 \item[a)] If $p(0)\leq -C(t)$ the optimal control is $\omega_a^*(t)=0$
 \item[b)]    If $p(\delta)\leq -C(t)<p(0)$, the optimal control is given by the minimal root of equation \eqref{bvl}
 \item[c)] If $-C(t)<p(\delta)$, equation \eqref{bvl} has no solution on $\mc C_a$, 
and so~\eqref{eq:optprogstabilityleft} admits no solution
\end{itemize}
Figure~\ref{fig:curves} on the left illustrates the three cases a), b), c).\\
\begin{figure}
\centering
\includegraphics[width=0.495\linewidth]{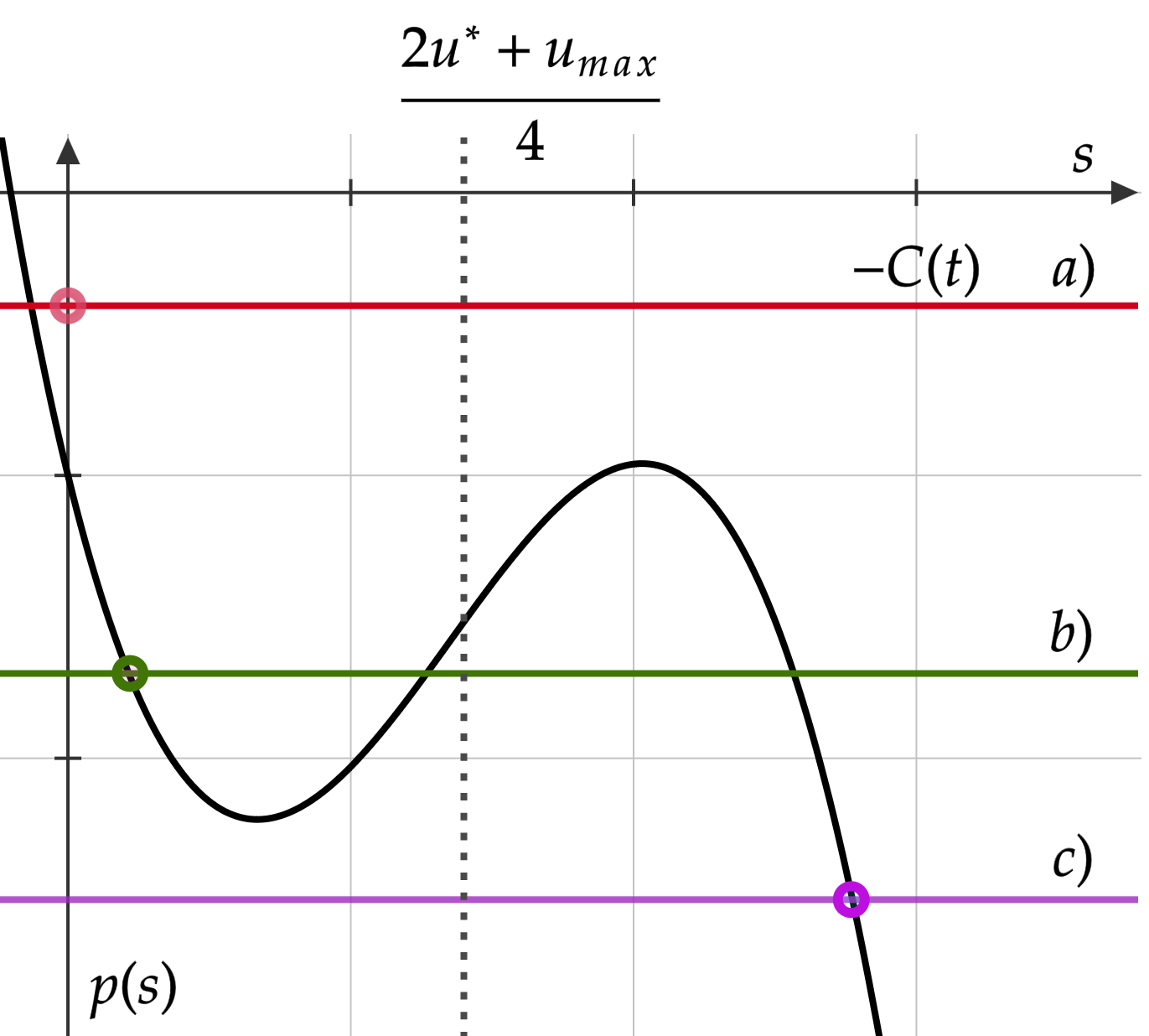}\hfill
\includegraphics[width=0.495\linewidth]{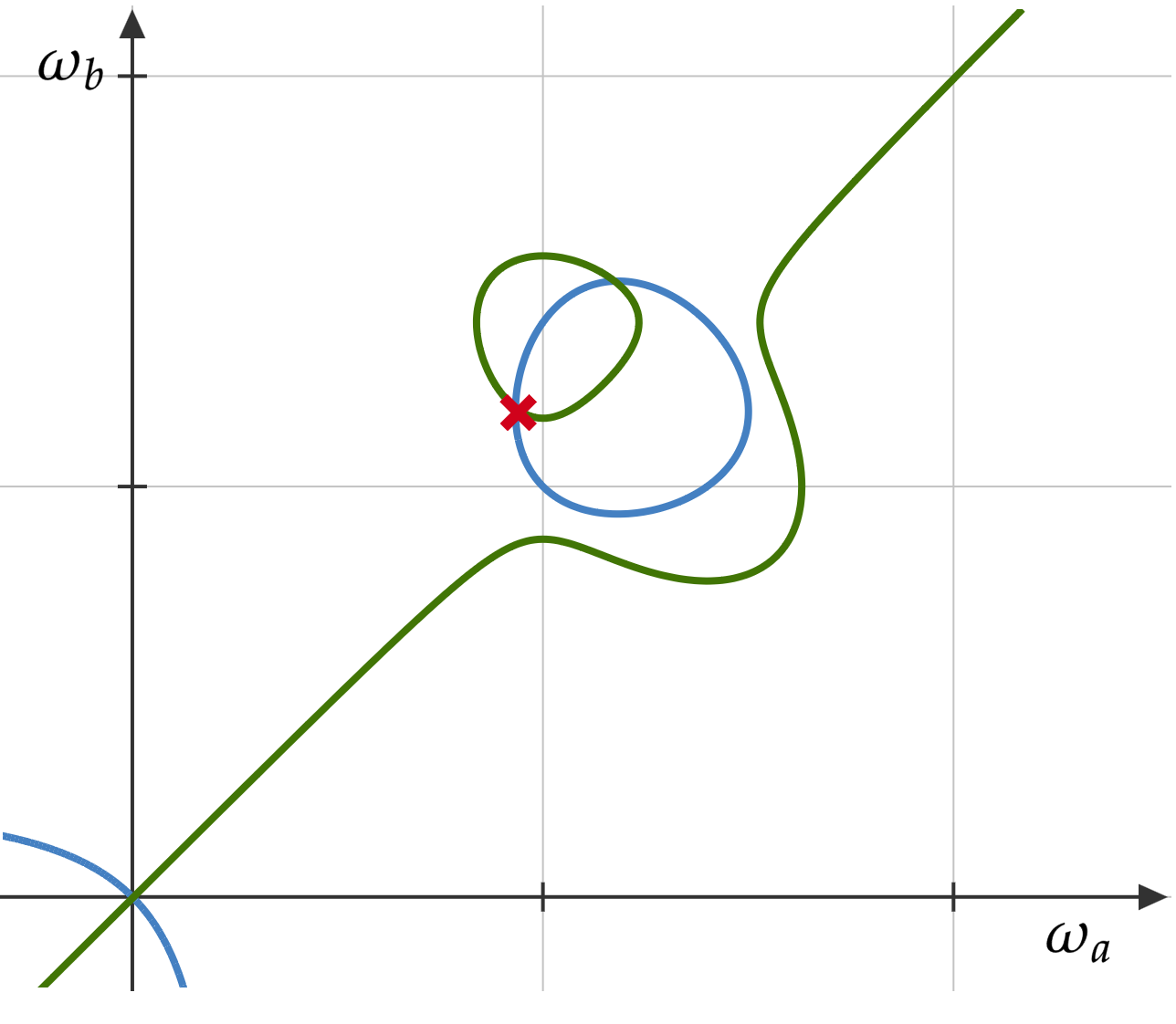}
\caption{On the left, a graphical representation of the roots of~\eqref{bvl} associated with~\eqref{eq:optprogstabilityleft}, case $u^*\neq \hat{u}$, for three different values of $-C(t)$; on the right, a possible configuration of the curves for problem~\eqref{eq:optprogstabilityleftright}, first (blue) and second (green) curves from~\eqref{eq:curv}.}
\label{fig:curves}
\end{figure}%
\noindent{\bf Case } $u^* = \hat{u}$: then $p'(s) < 0$ on $[0, \umax]\setminus\{\hat{u}\}$, $p'(\hat{u}) = 0$. 
Thus, equation~\eqref{bvl} has at most one root. In particular,
\begin{itemize}
    \item[a)] If $p(0)\leq -C(t)$ the optimal control is $\omega_a^*(t)=0$
    \item[b)] If $p(\frac{2 u^* + \umax}{4})\leq -C(t) < p(0)$, the optimal control is given by the unique root of equation \eqref{bvl} on $\left( 0, \frac{2 u^* + \umax}{4}\right]$
    \item[c)] If  $-C(t)<p(\frac{2 u^* + \umax}{4})$, equation \eqref{bvl} has no solution on $\mc C_a$, 
 and so~\eqref{eq:optprogstabilityleft} admits no solution
\end{itemize}
We now consider~\eqref{eq:optprogstabilityright}
\begin{equation*}
    \min_{\omega_b(t)\in\mc C_b} \left\{ \omega_b(t)^2 ~:~ g(u(t,a),\omega_b(t))\le -C(t)\right\}.
\end{equation*}
Let $q(z):= g(s,z)$, and notice that $q'(z) = \frac{1}{\hat{u}}(z-u^*)(z-\hat{u})$.\\
\noindent{\bf Case } $u^*\neq \hat{u}$: then $q'(z) > 0$ on $[0,\delta)\cup (\gamma,\umax]$ and $q'(z) < 0$ on $(\delta,\gamma)$. Therefore, 
\begin{itemize}
    \item[a)] If $-C(t)\ge q\left(\frac{2u^* + \umax}{4}\right)$, $\omega_b^*(t) = \frac{2u^* + \umax}{4}$ is the optimal solution to~\eqref{eq:optprogstabilityright}
    \item[b)] If $q(\gamma) \le -C(t) < q\left(\frac{2u^* + \umax}{4}\right)$, then $\omega_b^*(t)$ is the minimal root of $q(z) = -C(t)$ on $\left(\frac{2u^* + \umax}{4}, \umax\right]$
    \item[c)] if $-C(t) < q(\gamma)$, then there exists no solution to~\eqref{eq:optprogstabilityright}
\end{itemize}
{\bf Case} $u^* = \hat{u}$: then $q'(z) = \frac{1}{\hat{u}}(z-\hat{u})^2\ge 0$, and thus
\begin{itemize}
   \item[a)] If $-C(t)\ge q\left(\frac{2u^* + \umax}{4}\right)$, then $\omega_b^*(t) = \frac{2u^* + \umax}{4}$ is the optimal solution to~\eqref{eq:optprogstabilityright}
   \item[b)]  If $-C(t) < q\left(\frac{2u^* + \umax}{4}\right)$, then there exists no solution to~\eqref{eq:optprogstabilityright}
 \end{itemize}
We now look at~\eqref{eq:optprogstabilityleftright}:
\begin{equation*}
\min_{ \tiny\begin{aligned}&\omega_a(t)\in\mc C_a\\ &\omega_b(t)\in\mc C_b\end{aligned}} \left\{ \omega_a(t)^2+\omega_b(t)^2 ~:~ g(\omega_a(t),\omega_b(t))\le -C(t)\right\}.
\end{equation*}
By taking the Lagrangian, we can observe that 
critical points are  solutions of the system 
\begin{equation}\label{eq:curv}
    \begin{cases}
       sz\big( f'(s)+f'(z)\big)=u^* \big(sf'(z)+zf'(s)\big)\\ 
g(\omega_a(t),\omega_b(t)) +C(t) =0.
   \end{cases}
\end{equation}
An admissible optimal control, if any, is the point of minimal norm in the intersection of the curves defined by the equations in  \eqref{eq:curv} (see right figure in Fig.~\ref{fig:curves}). 
Due to the strong nonlinearity of these equations, we rely on numerical evidence to assert that an optimal control exists only if
$$C(t) = \alpha\Big(\frac{1}{2} \int_a^b (u(t,x) - u^*)^2 \di x \Big)\approx 0$$
meaning that the solution of the boundary value problem is nearly at the target state.

\subsection{Invariance-preserving Controllers}\label{sec:invariance-controller}
We first consider the existence of solutions to the invariance problem~\eqref{eq:optprogsafetyleft}, which reads as
\begin{equation*}
    \min_{\omega_a(t)\in\mc I_a} \left\{ \omega_a(t)^2 ~:~ k(\omega_a(t),u(t,b)) \le D(t) \right\}
\end{equation*}
The optimal solution is either $\omega_a^*(t)=0$, provided that it satisfies the constraint, or it lies on the boundary of the region defined by the constraint, i.e.,
\begin{equation}\label{eq:exoptinvariance}
k(\omega_a^*(t),u(t,b)) = D(t).
\end{equation}
Let $\ell(s) := k(s, u(t,b))$. It is easy to check that $\ell'(s) > 0$ on $\mc I_a$. So, if $\ell(0)\le D(t)$, then $\omega_a^*(t)=0$ is the optimal solution to~\eqref{eq:optprogsafetyleft} and if $D(t) < \ell(0)$, then there exists no solution to~\eqref{eq:optprogsafetyleft}.

Let us now consider problem~\eqref{eq:optprogsafetyright}: 
\begin{equation}
    \min_{\omega_b(t)\in\mc I_b} \left\{ \omega_b(t)^2 ~:~ k(u(t,a),\omega_b(t))\le D(t)\right\}.
\end{equation}
In this case, the optimal solution is $\omega_b^*(t)=\frac{\umax}{4}$, provided that it satisfies the constraint, or it lies on the boundary of the region defined by the constraint
\begin{equation}\label{eq:exoptinvarianceb}
k(u(t,a),\omega_b^*(t)) = D(t).
\end{equation}
Let $\rho(z) := k(u(t,a),z)$. Recalling that $\hat u$ is the unique point of maximum of the flux $f$ in \eqref{eq:flux}, we can then check that $\rho'(z) < 0$ on $[0,\hat{u})$ and $\rho'(z) > 0$ on $(\hat{u},\umax]$. Therefore,~\eqref{eq:exoptinvarianceb} has at most two roots on $\mc I_b$, and we conclude that:
if $D(t) \ge \rho\left(\frac{\umax}{4}\right)$, then $\omega_b^*(t)=\frac{\umax}{4}$ is the optimal solution to~\eqref{eq:optprogsafetyright};
if $\rho\left(\hat{u}\right) \le D(t) < \rho\left(\frac{\umax}{4}\right) $, then the optimal solution $\omega_b^*(t)$ to~\eqref{eq:optprogsafetyright} is given by the solution to~\eqref{eq:exoptinvarianceb} of minimal norm on $\mc I_b$, and if $D(t) < \rho\left(\hat{u}\right)$, then there exists no solution to~\eqref{eq:optprogsafetyright}.

We skip here the details of problem~\eqref{eq:optprogsafetyleftright}, which can be treated similarly to~\eqref{eq:optprogstabilityleftright}.

\begin{remark}
   The optimization problems \eqref{eq:optprogstabilityleftright} and \eqref{eq:optprogsafetyleftright} involving control at both boundaries might not admit a solution. The counterintuitive result is that our formulation lends itself to the optimization of single boundary controls insofar as it results in wider feasibility intervals. In the next section, we will propose a modified formulation to achieve stability and invariance using both boundary controls.
\end{remark}

\section{Existence of Stabilizing and Invariance-preserving Controllers}\label{sec:mixedpb}

We now consider the compound problem that aims to ensure stability from one boundary point and invariance from the other one. With the insights gained in the previous section when studying the feasibility of \eqref{eq:optprogstabilityleftright} and \eqref{eq:optprogsafetyleftright}, we reformulate~\eqref{eq:optprogstabilityleft}-\eqref{eq:optprogsafetyright} in the following relaxed version:
\begin{equation}
    \begin{aligned}
        \min_{\omega_a(t)\in\mc C_a} \big\{ \omega_a(t)^2 ~:~
         & g(\omega_a(t),z_1) \le - C(t),\\[-0.5ex]
         & k(\omega_a(t),z_2)\le D(t),\,z_1,z_2 \in \mc I_b \big\}
    \end{aligned}
     \label{eq:mixoptimpbA}
\end{equation}
and 
\begin{equation}\label{eq:mixoptimpbB}
    \begin{aligned}
        \min_{\omega_b(t) \in \mc I_b} \big\{ \omega_b(t)^2 ~:~ & g(s_1,\omega_b(t)) \le - C(t),\\[-0.5ex]
        & k(s_2,\omega_b(t))\le D(t),\,s_1,s_2 \in \mc C_a\}.
    \end{aligned}
\end{equation}
Notice that we solve independently the two optimization problems yet ensuring that each satisfies both constraints within the admissibility class of the other.
\begin{theorem}\label{thm:mixedpb}
There exist values of the parameters $u^*$, $\bar{u}\le \umax$ and class $\mc K$ functions $\alpha, \beta$ such that the optimization problem~\eqref{eq:mixoptimpbA}-\eqref{eq:mixoptimpbB} is feasible.
\end{theorem}
\begin{proof}
We first focus on~\eqref{eq:mixoptimpbA}. It is clear that $\omega_a^*(t) = 0$ is the optimal solution provided that the constraints are satisfied. If not, then $\omega_a^*(t)$ lies at the boundary of the constraint set, which is given by either
\begin{equation}\label{V1}
\begin{cases}
  g(\omega_a^*(t),z_1) = -C(t), & z_1 \in \mc I_b,
\\[0.3ex]
 k(\omega_a^*(t),z_2)\le D(t), & z_2\in \mc I_b,
\end{cases}
\end{equation}
or 
\begin{equation}\label{V2}
\begin{cases}
  g(\omega_a^*(t),z_1) \le -C(t), & z_1 \in \mc I_b,
\\[0.3ex]
 k(\omega_a^*(t),z_2) = D(t), & z_2\in \mc I_b.
\end{cases}
\end{equation}

Denote by $\mc U$ the set of all admissible solutions to both~\eqref{V1} and~\eqref{V2} on $\mc C_a$. 
We now discuss each case separately. Reasoning as in \ref{sec:stabcontrol} with the monotonicity of $p(s):= g(s,z)$, $z\in \mc I_b$, we characterize the solutions to~\eqref{V1} as follows.\\
{\bf Case} $u^*\neq \hat{u}$: then $p'(s) < 0$ on  $[0,\delta)\cup(\gamma, \umax]$ and $p'(s) > 0$ on $(\delta,\gamma)$. Thus,
\begin{itemize}
    \item[a)] If
    $g(0,z_1)\leq -C(t)$ for some $z_1\in \mc I_b$, 
    then $\omega_a^*(t)=0$ is an admissible solution of~\eqref{V1} provided that $k(\omega_a^*(t),z_2)\le D(t)$ for some $z_2\in \mc I_b$. If so, we store it in $\mc U$
    \item[b)] If $g(\delta,z_1)\leq -C(t)< g(0,z_1)$ for some $z_1\in \mc I_b$, then among all the roots $s^*$ to $p(s) = -C(t)$ select the one of minimal norm that also satisfies $k(s^*,z)\le D(t)$ for some $z\in \mc I_b$, if any, and store it in $\mc U$
    \item[c)] If $-C(t)<g(\delta,z)$ for all $z\in \mc I_b$, then~\eqref{V1} has no solution and we proceed searching for a solution to~\eqref{V2}
\end{itemize}
\begin{figure*}
    \centering
    \subfloat[$t=0.3$ s]{\label{subfig:snap1}\includegraphics[trim={0 0 2cm 0},clip,width=0.16\linewidth]{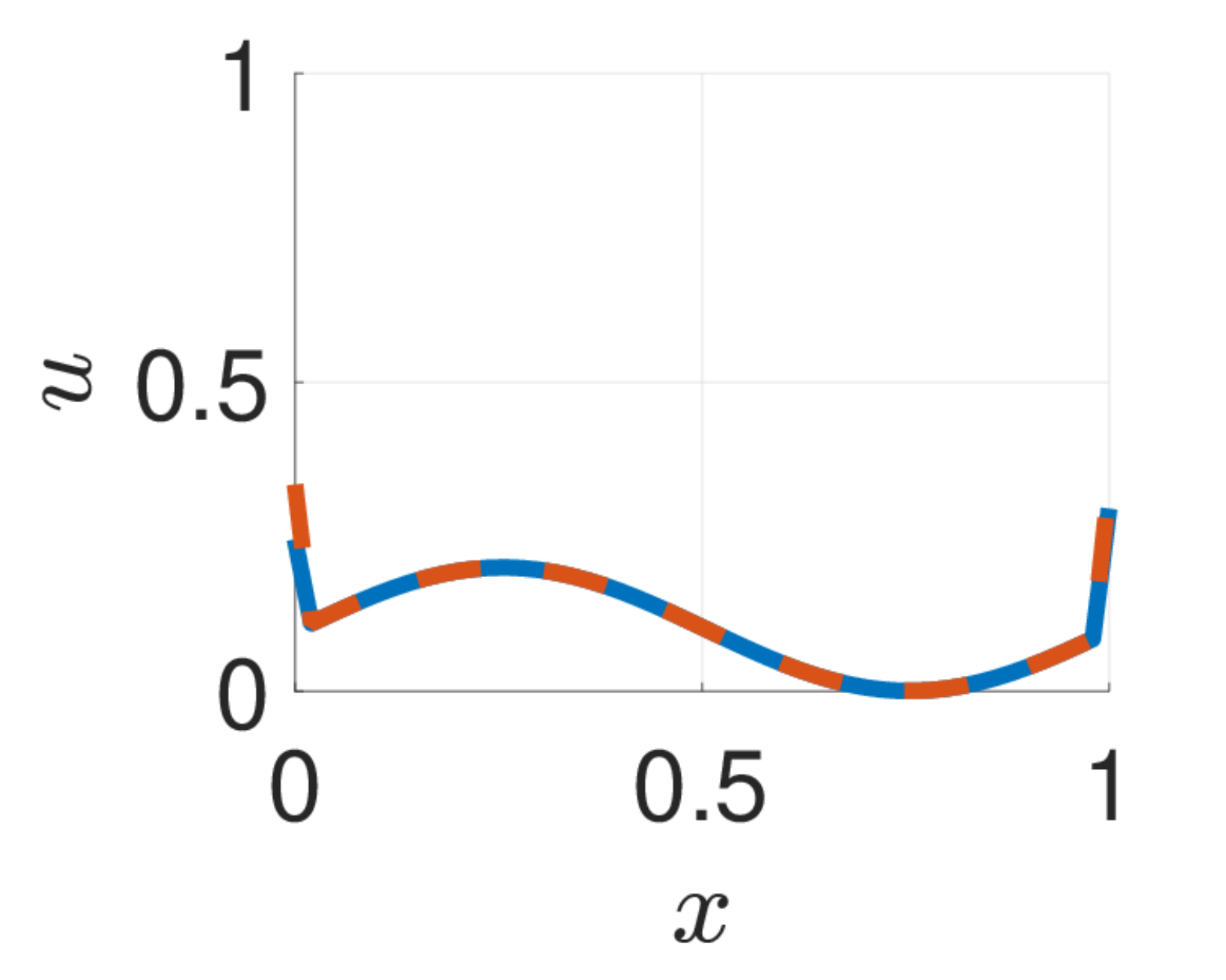}}\hfill
    \subfloat[$t=1.5$ s]{\label{subfig:snap2}\includegraphics[trim={0 0 2cm 0},clip,width=0.16\linewidth]{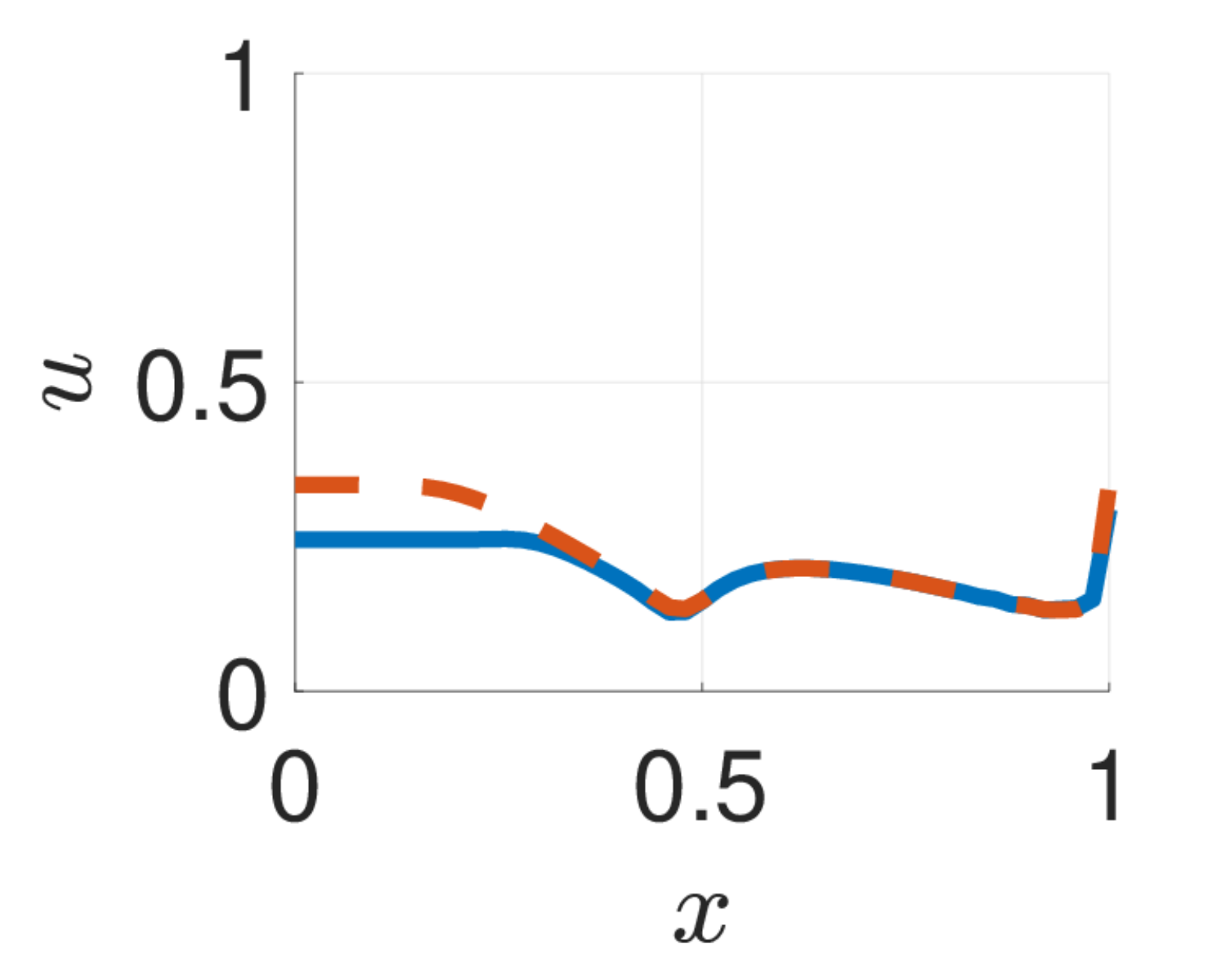}}\hfill
    \subfloat[$t=3.0$ s]{\label{subfig:snap3}\includegraphics[trim={0 0 2cm 0},clip,width=0.16\linewidth]{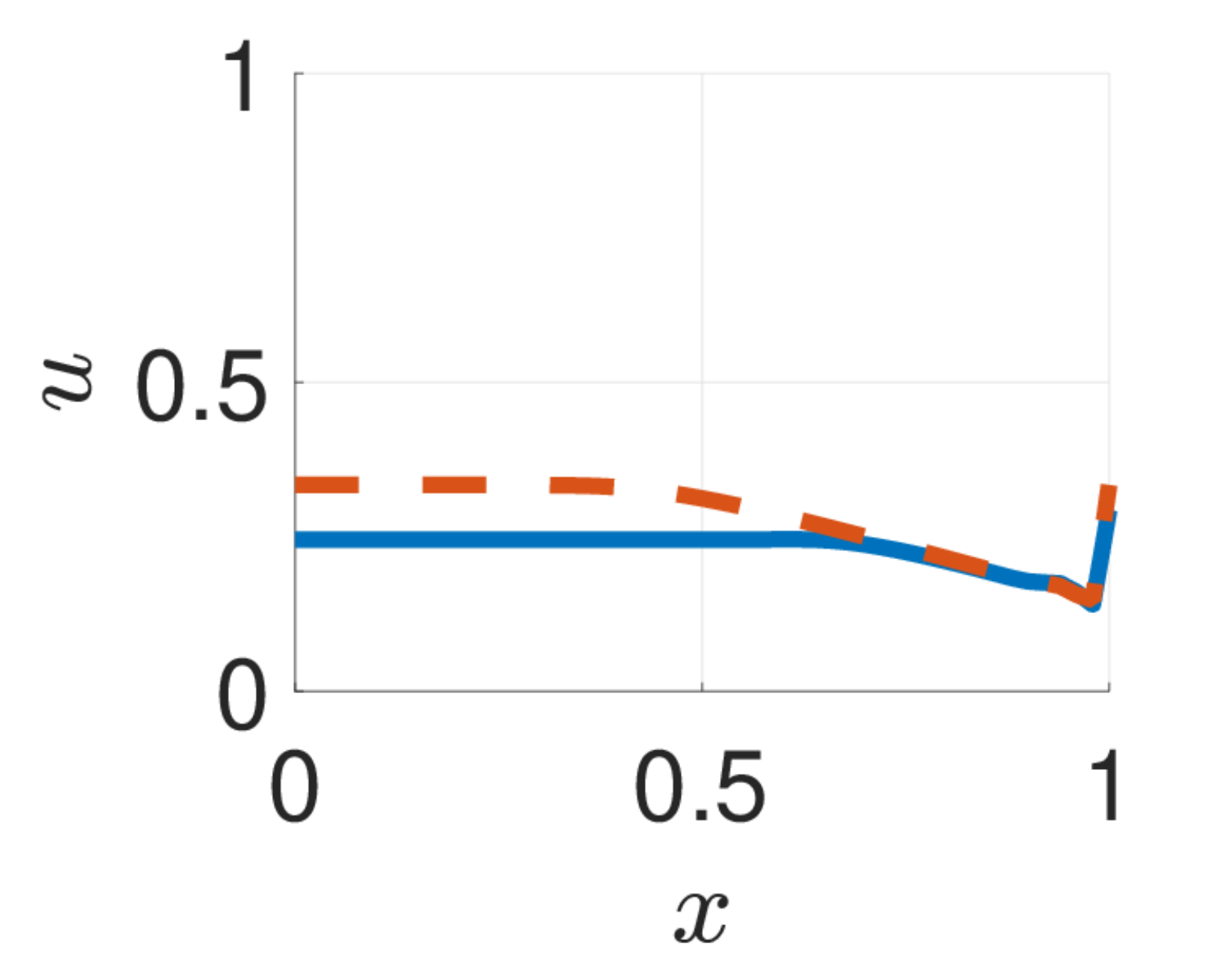}}\hfill
    \subfloat[$t=4.5$ s]{\label{subfig:snap4}\includegraphics[trim={0 0 2cm 0},clip,width=0.16\linewidth]{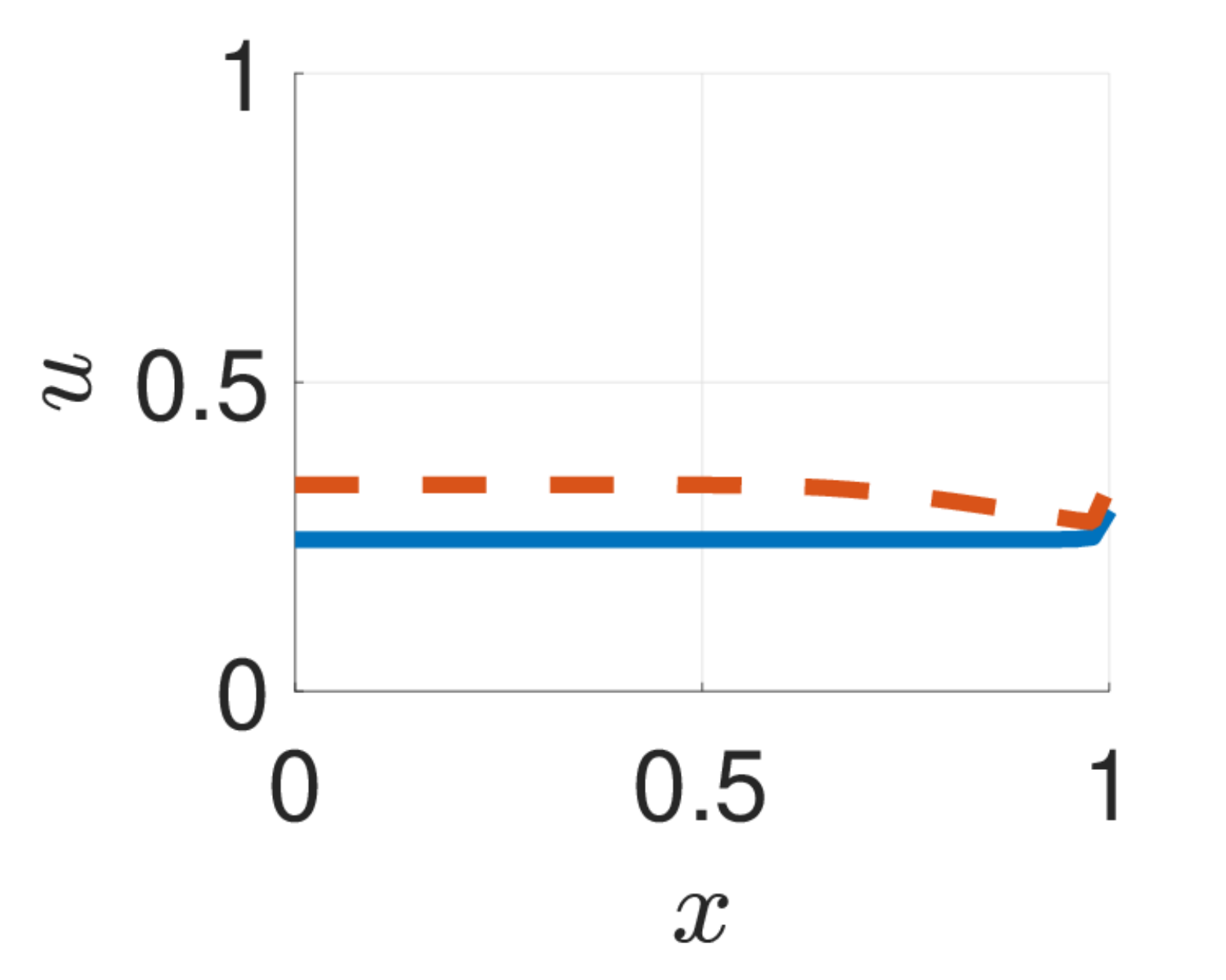}}\hfill
    \subfloat[$t=15$ s]{\label{subfig:snap5}\includegraphics[trim={0 0 2cm 0},clip,width=0.16\linewidth]{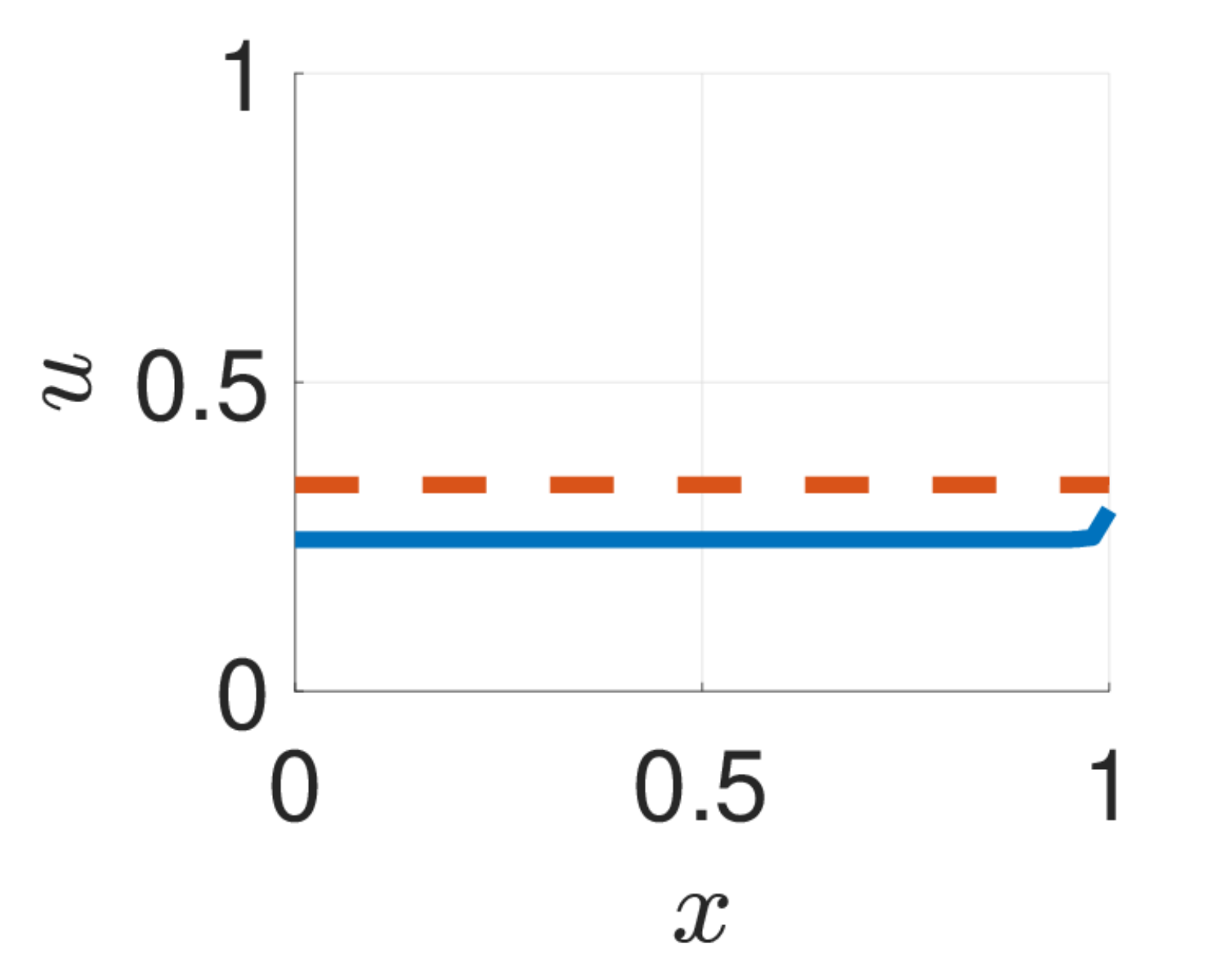}}\hfill
    \subfloat[$t=30$ s]{\label{subfig:snap6}\includegraphics[trim={0 0 2cm 0},clip,width=0.16\linewidth]{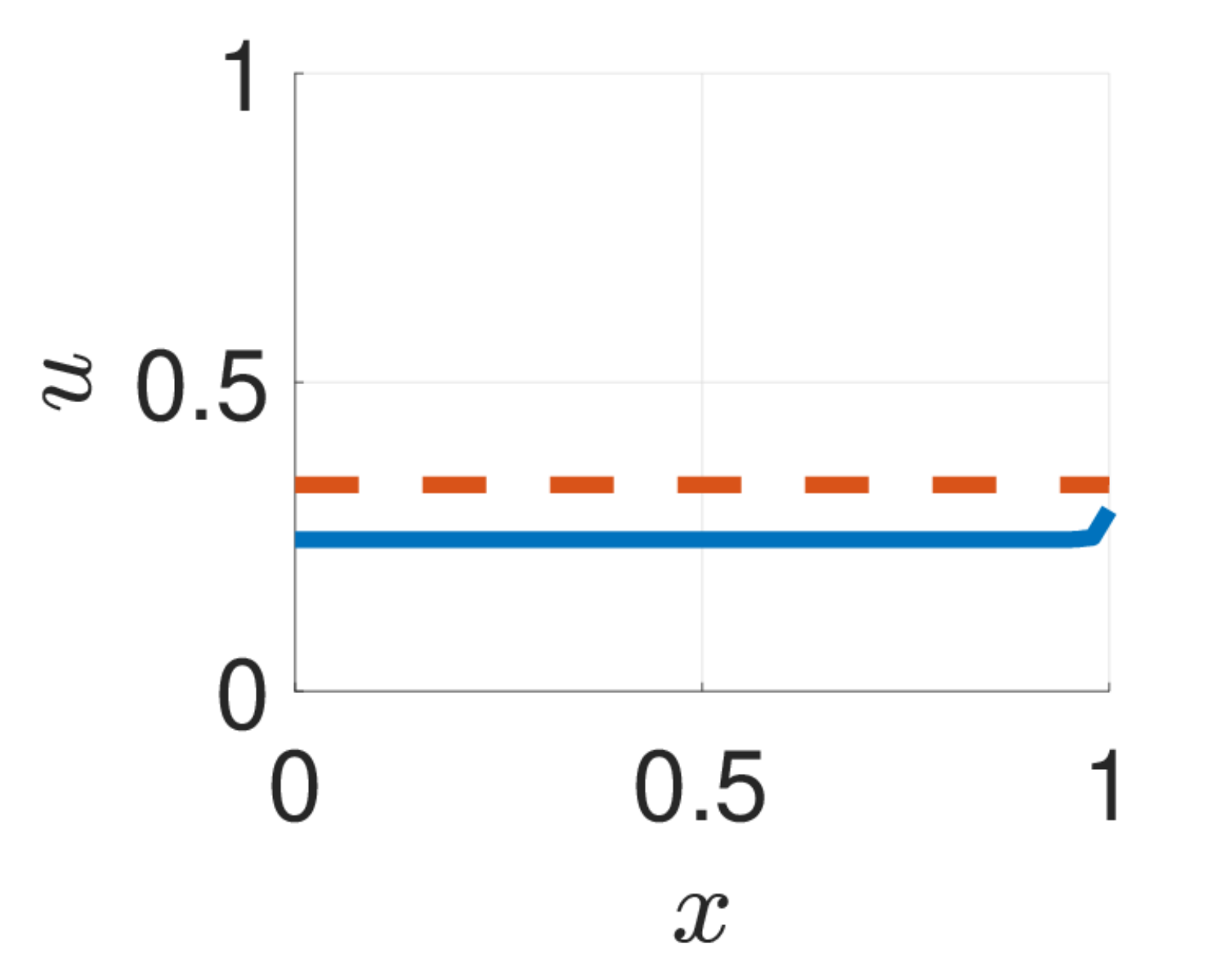}}\\
    \subfloat[Time evolution of {$V(u(t))$ in green and $B(u(t))$ in red}.]{\label{subfig:lyap_barr}\includegraphics[trim={3.5cm 0 3.5cm 0},clip,width=0.495\linewidth]{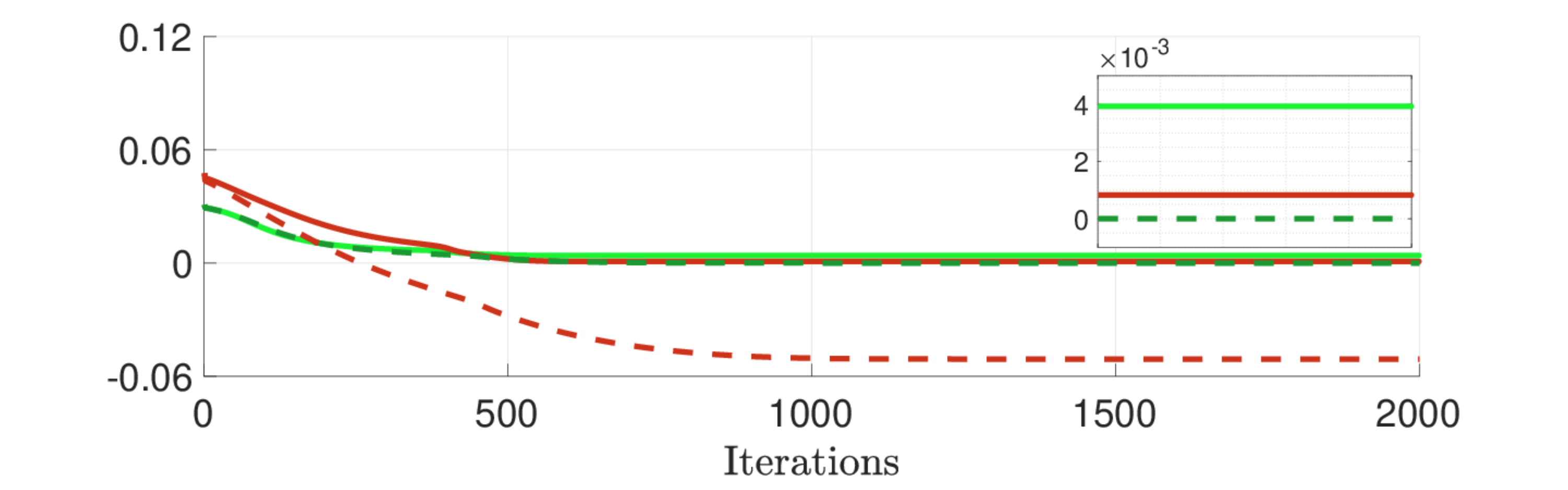}}\hfill
    \subfloat[Time evolution of {$\omega_a(t)$ in blue and $\omega_b(t)$ in orange.}]{\label{subfig:inputs}\includegraphics[trim={3.5cm 0 3.5cm 0},clip,width=0.495\linewidth]{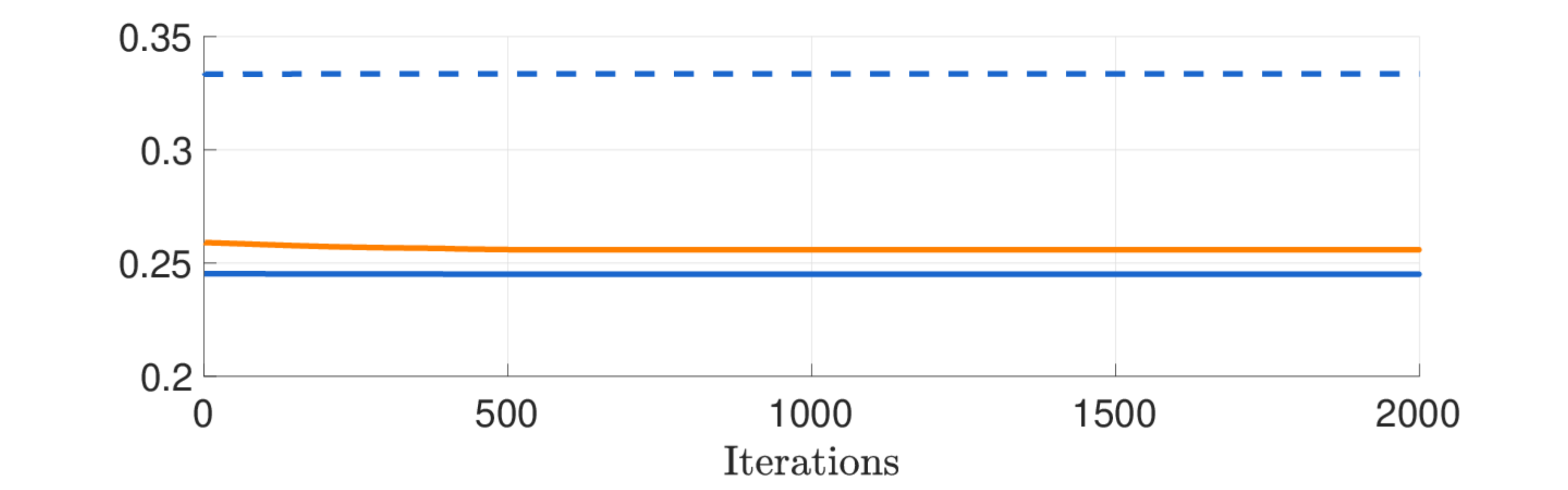}}
    \caption{Figures~\ref{subfig:snap1}--\ref{subfig:snap6} show the time evolution of the state, starting from a sinusoidal initial condition and controlled to achieve only stability (red dashed line) and stability with invariance (blue solid line). In the former case, the left boundary is computed using \eqref{eq:optprogstabilityleft}, while the right boundary is set to $u^*=1/3$. In the latter case, the boundary controls are evaluated using \eqref{eq:optprogstabilityleft}-\eqref{eq:optprogsafetyright}. Fig.~\ref{subfig:lyap_barr} shows that invariance is violated ($B$ in red dashed line becomes negative) in the case of stability-only control, whereas $B$ remains positive when invariance is controlled.}
    \label{fig:simulations}
\end{figure*}
{\bf Case} $u^* = \hat{u}$: $p'(s) < 0$ on $[0, \umax]\setminus\{\hat{u}\}$, $p'(\hat{u}) = 0$. 
Thus,
\begin{itemize}
    \item[a)]  If {$g(0,z_1)\leq -C(t)$ for some $z_1\in \mc I_b$,} then $\omega_a^*(t)=0$ is an admissible solution to~\eqref{V1} provided that $k(\omega_a^*(t),z_2)\le D(t)$ for some $z_2\in \mc I_b$, and we store it in $\mc U$
    \item[b)] If $g(\frac{2 u^* + \umax}{4},z_1)\leq -C(t)<g(0,z_1)$ for some $z_1\in \mc I_b$, an admissible solution to~\eqref{V1} is given by the unique root $s^*$ of $p(s) = -C(t)$ on $\left(0, \frac{2 u^* + \umax}{4}\right]$ provided that $k(s^*,z_2)\le D(t)$ for some $z_2\in \mc I_b$, and we store it in $\mc U$
    \item[c)] If $-C(t)<g(\frac{2 u^* + \umax}{4},z)$ for all $z\in \mc I_b$, equation  $p(s) = -C(t)$ has no root on $\mc C_a$, and so~\eqref{V1} admits no solution and we proceed searching for a solution to~\eqref{V2}
\end{itemize}
We now consider the case~\eqref{V2}. Arguing as in \ref{sec:invariance-controller} with the monotonicity of $\ell(s) := k(s, z)$, $z\in \mc I_b$, we notice that $\ell'(s) > 0$ on $\left[0,\umax/2\right)$ and $\ell'(s)< 0$ on $\left(\umax/2, \umax\right]$. 
Since $\frac{2 u^* + \umax}{4} < \frac{\umax}{2}$, we get that
\begin{itemize}
\item[a)] If $D(t) \ge k(0,z_1)$ for some $z_1\in \mc I_b$, then $\omega_a^*(t)=0$ is an admissible solution of~\eqref{V2} provided that $g(\omega_a^*(t),z_2)\le -C(t)$ for some $z_2\in \mc I_b$, and we store it in $\mc U$
\item[b)] If $D(t) <  k(0,z_1)$ for all $z_1\in \mc I_b$,~\eqref{V2} has no solution
\end{itemize}
At this stage, the optimal solution to problem~\eqref{eq:mixoptimpbA} is given by the element $s^*$ in $\mc U$ of minimal norm. 

In a similar manner, we can solve the problem~\eqref{eq:mixoptimpbB}. It is clear that $\omega_b^*(t) = \frac{\umax}{4}$ is the optimal solution provided that the constraints are satisfied. If not, then $\omega_b^*(t)$ lies at the boundary of the constraint set, which is given by either 
\begin{equation}\label{V3}
\begin{cases}
  g(s_1,\omega_b^*(t)) = -C(t), & s_1 \in \mc C_a
\\[0.3ex]
 k(s_2,\omega_b^*(t))\le D(t), & s_2\in \mc C_a
\end{cases}
\end{equation}
or
\begin{equation}\label{V4}
\begin{cases}
  g(s_1,\omega_b^*(t)) \le -C(t), & s_1 \in \mc C_a
\\[0.3ex]
 k(s_2,\omega_b^*(t)) = D(t), & s_2\in \mc C_a.
\end{cases}
\end{equation}
Denote by $\mc W$ the set of all admissible solutions to both~\eqref{V3} and~\eqref{V4} on $\mc I_b$. 
In the case~\eqref{V3}, arguing as in Section~\ref{sec:stabcontrol} with the monotonicity of $q(z):= g(s,z)$, $s\in\mc C_a$, we have the following.\\
{\bf Case} $u^*\neq \hat{u}$: then $q'(z) > 0$ on $[0,\delta)\cup (\gamma,\umax]$ and $q'(z) < 0$ on $(\delta,\gamma)$. Thus, if $u^*\le \umax/4$,
\begin{itemize}
    \item[a)] If $g(s_1,\umax/4)\leq -C(t)$ for some $s_1\in \mc C_a$, then $\omega_b^*(t)=\umax/4$ is an admissible solution to~\eqref{V3} provided that $k(s_2,\omega_b^*(t))\le D(t)$ for some $s_2\in \mc C_a$. If so, we store it in $\mc W$
    \item[b)] If $g(s_1,\hat{u})\leq -C(t)< g(s_1,\umax/4)$ for some $s_1\in \mc C_a$, then among all the roots $z^*$ to $q(z) = -C(t)$ select the one of minimal norm that also satisfies $k(s,z^*)\le D(t)$ for some $s\in \mc C_a$, if any, and store it in $\mc W$
    \item[c)] If $-C(t)<g(s,\hat{u})$ for all $s\in \mc C_a$, then~\eqref{V3} does not admit any solution on $\mc I_b$ and we proceed searching for a solution to the case~\eqref{V4}
\end{itemize}
If $\umax/4 < u^* < \hat{u}$ or $\hat{u} < u^*$, then
\begin{itemize}[itemindent=-0.5em]
    \item[a)] If $g(s_1,\umax/4)\leq -C(t)$ for some $s_1\in \mc C_a$, $\omega_b^*(t)=\umax/4$ is an admissible solution of~\eqref{V3} provided that $k(s_2,\omega_b^*(t))\le D(t)$ for some $s_2\in \mc C_a$. If so, we store it in $\mc W$
    \item[b)] If $g(s_1,\gamma)< g(s_1,\umax/4)$ and $g(s_1,\gamma)\leq -C(t)< g(s_1,\umax/4)$ for some $s_1\in \mc C_a$, then among the two roots $z^*$ to $q(z) = -C(t)$ select the one of minimal norm that also satisfies $k(s_2,z^*)\le D(t)$ for some $s_2\in \mc C_a$, if any, and store it in $\mc W$
    \item[c)] If $-C(t)< \min\{g(s,\gamma),g(s,\umax/4)\}$ for all $s\in \mc C_a$, then~\eqref{V3} does not admit any solution on $\mc I_b$ and we proceed searching for a solution to the case~\eqref{V4}
\end{itemize}

{\bf Case} $u^* = \hat{u}$: since $q'(z) = \frac{1}{\hat{u}}(z-\hat{u})^2\ge 0$, thus
\begin{itemize}
    \item[a)]  If $g(s_1,\umax/4)\leq -C(t)$ for some $s_1\in \mc C_a$, then $\omega_b^*(t)=\umax/4$ is an admissible solution to~\eqref{V3} provided that $k(s_2,\omega_b^*(t))\le D(t)$ for some $s_2\in \mc C_a$. If so, we store it in $\mc W$
    \item[b)] If $-C(t)< g(s,\umax/4)$ for all $s\in \mc C_a$, then~\eqref{V3} does not admit any solution on $\mc I_b$ and we proceed searching for a solution to the case~\eqref{V4}
\end{itemize}
We now focus on the case~\eqref{V4}. Reasoning as in Section~\ref{sec:invariance-controller} with the monotonicity of $\rho(z):= k(s,z)$, $s\in \mc C_a$, we have the following.
\begin{itemize}
    \item[a)] If $k(s_2,\umax/4)\leq D(t)$ for some $s_2\in \mc C_a$, then $\omega_b^*(t)=\umax/4$ is an admissible solution of~\eqref{V4} provided that $g(s_1,\omega_b^*(t))\le D(t)$ for some $s_1\in \mc C_a$. If so, we store it in $\mc W$
    \item[b)] If $k(s_2,\hat{u})\leq D(t)< k(s_2,\umax/4)$ for some $s_2\in \mc C_a$, then among the two roots $z^*$ to $\rho(z) = D(t)$ select the one of minimal norm that also satisfies $g(s_1,z^*)\le -C(t)$ for some $s_1\in \mc C_a$, if any, and store it in $\mc W$
    \item[c)] If $D(t)< k(s,\hat{u})$ for all $s\in \mc C_a$, then~\eqref{V4} does not admit any solution on $\mc I_b$
\end{itemize}
At this point, the optimal solution to problem~\eqref{eq:mixoptimpbB} is given by the element $z^*$ in $\mc W$ of minimal norm.

Finally, the solution to the compound problem is given by the pair $(\omega^*_a(t),\omega^*_b(t))$ where $\omega^*_a(t)$ solves~\eqref{eq:mixoptimpbA} and $\omega^*_b(t)$ is the solution to~\eqref{eq:mixoptimpbB}, if both $\mc U$ and $\mc W$ are nonempty.
\end{proof}

\section{Simulations}

In this section, we illustrate the behavior of system~\eqref{CL}-\eqref{BVs} and compare the solution to the stability-only problem (8) from Section III with the solution to the stabilizing and invariance-preserving problem \eqref{eq:mixoptimpbA}-\eqref{eq:mixoptimpbB} from Section IV. In this way, in the first instance, we control the left boundary to achieve stability, while in the second case, we also control the right boundary in order to achieve invariance.
The parameters used to simulate the system are as follows: $a=0$, $b=1$, $\umax=1$, $u^*=1/3$ in \eqref{eq:clf}, and $\bar u = 1/4$ in \eqref{eq:bcbf}. We simulate the system over the time interval $[0,30]$ s.

The snapshots of the two simulations are shown in Fig.~\ref{fig:simulations}. Figures~\ref{subfig:snap1}--\ref{subfig:snap6} report the time evolution of the state of the system \eqref{CL}, where 1 iteration corresponds to $0.015$ s. The red dashed line depicts the case where only $\omega_a$ is controlled, using \eqref{eq:optprogstabilityleft}, while the blue solid line represents the system state behavior under the stability and invariance-preserving boundary controllers achieved via $\omega_a$ and $\omega_b$ computed by \eqref{eq:mixoptimpbA} and \eqref{eq:mixoptimpbB}, respectively. Starting from a sinusoidal initial condition, in the first case, the stabilizing controller leads the state to settle at $u^*=1/3$, while in the second case, the invariance constraint that the $L^2$-norm of the state has to be less than $\bar u=1/4$ is achieved. The invariance is obtained at the expense of stability, as can be quantitatively seen in Fig.~\ref{subfig:lyap_barr}. Lower values of the Lyapunov function (green) correspond to the state $u$ being closer to the desired value $u^*$ in the $L^2$-norm sense, and positive values of the BCBFal (red) mean that the norm of the state is less than $\bar u$ (see, in particular, the zoomed popout axes in Fig.~\ref{subfig:lyap_barr}). The regularity of the controllers to achieve the desired behaviors can be appreciated in Fig.~\ref{subfig:inputs}, corroborating the choice made in Section~\ref{sec:mixedpb} of solving \eqref{eq:mixoptimpbA} and \eqref{eq:mixoptimpbB} in place of \eqref{eq:optprogstabilityleft}-\eqref{eq:optprogsafetyright}.

\section{Conclusions}

In this letter, we introduced boundary control barrier functionals for PDEs, derived constraints for state invariance, and formulated a convex optimization-based control policy, providing sufficient conditions for the existence of optimal boundary controllers. Although our analysis focused on traffic flow dynamics, the proposed approach is more general and can be extended to address stabilization and invariance in a wider class of PDEs. In particular, future work will explore its application to more general conservation laws as well as the regularity of the resulting controller. 

\bibliographystyle{bib/IEEEtran}
\bibliography{bib/IEEEabrv,bib/references}

\begin{thebibliography}{10}
\providecommand{\url}[1]{#1}
\csname url@rmstyle\endcsname
\providecommand{\newblock}{\relax}
\providecommand{\bibinfo}[2]{#2}
\providecommand\BIBentrySTDinterwordspacing{\spaceskip=0pt\relax}
\providecommand\BIBentryALTinterwordstretchfactor{4}
\providecommand\BIBentryALTinterwordspacing{\spaceskip=\fontdimen2\font plus
\BIBentryALTinterwordstretchfactor\fontdimen3\font minus
  \fontdimen4\font\relax}
\providecommand\BIBforeignlanguage[2]{{%
\expandafter\ifx\csname l@#1\endcsname\relax
\typeout{** WARNING: IEEEtran.bst: No hyphenation pattern has been}%
\typeout{** loaded for the language `#1'. Using the pattern for}%
\typeout{** the default language instead.}%
\else
\language=\csname l@#1\endcsname
\fi
#2}}

\bibitem{vazquez2024backstepping}
\BIBentryALTinterwordspacing
R.~Vazquez, J.~Auriol, F.~Bribiesca-Argomedo, and M.~Krstic, ``{Backstepping
  for Partial Differential Equations},'' 2025. [Online]. Available:
  \url{https://arxiv.org/abs/2410.15146}
\BIBentrySTDinterwordspacing

\bibitem{daudin2023optimal}
S.~Daudin, ``Optimal control of the {F}okker-{P}lanck equation under state
  constraints in the {W}asserstein space,'' \emph{Journal de Math{\'e}matiques
  Pures et Appliqu{\'e}es}, vol. 175, pp. 37--75, 2023.

\bibitem{ledyaev1999lyapunov}
Y.~S. Ledyaev and E.~D. Sontag, ``A {L}yapunov characterization of robust
  stabilization,'' \emph{Nonlinear Analysis-Series A Theory and Methods-Series
  B Real World Applications}, vol.~37, no.~7, pp. 813--840, 1999.

\bibitem{ames2019control}
A.~D. Ames, S.~Coogan, M.~Egerstedt, G.~Notomista, K.~Sreenath, and P.~Tabuada,
  ``Control barrier functions: Theory and applications,'' in \emph{2019 18th
  European control conference (ECC)}.\hskip 1em plus 0.5em minus 0.4em\relax
  IEEE, 2019, pp. 3420--3431.

\bibitem{hu2025boundary}
H.~Hu and C.~Liu, ``On the {B}oundary {F}easibility for {PDE} {C}ontrol with
  {N}eural {O}perators,'' \emph{Proceedings of Machine Learning Research vol
  vvv}, vol.~1, p.~27, 2025.

\bibitem{wang2023safe}
\BIBentryALTinterwordspacing
J.~Wang and M.~Krstic, ``Output-{P}ositive {A}daptive {C}ontrol of {H}yperbolic
  {PDE-ODE C}ascades,'' 2025. [Online]. Available:
  \url{https://arxiv.org/abs/2309.05596}
\BIBentrySTDinterwordspacing

\bibitem{lifourier}
Z.~Li, N.~B. Kovachki, K.~Azizzadenesheli, K.~Bhattacharya, A.~Stuart,
  A.~Anandkumar, \emph{et~al.}, ``Fourier {N}eural {O}perator for {P}arametric
  {P}artial {D}ifferential {E}quations,'' in \emph{International Conference on
  Learning Representations}, 2021.

\bibitem{LW}
M.~J. Lighthill and G.~B. Whitham, ``On kinematic waves {II}. {A} theory of
  traffic flow on long crowded roads,'' \emph{Proceedings of the royal society
  of london. series a. mathematical and physical sciences}, vol. 229, no. 1178,
  pp. 317--345, 1955.

\bibitem{R}
P.~I. Richards, ``Shock waves on the highway,'' \emph{Operations research},
  vol.~4, no.~1, pp. 42--51, 1956.

\bibitem{bardos1979first}
C.~Bardos, A.-Y. LeRoux, and J.-C. N{\'e}d{\'e}lec, ``First order quasilinear
  equations with boundary conditions,'' \emph{Communications in partial
  differential equations}, vol.~4, no.~9, pp. 1017--1034, 1979.

\bibitem{7509658}
S.~Blandin, X.~Litrico, M.~L. Delle~Monache, B.~Piccoli, and A.~Bayen,
  ``Regularity and {L}yapunov {S}tabilization of {W}eak {E}ntropy {S}olutions
  to {S}calar {C}onservation {L}aws,'' \emph{IEEE Transactions on Automatic
  Control}, vol.~62, no.~4, pp. 1620--1635, 2017.

\bibitem{bayen2022control}
A.~Bayen, M.~L. Delle~Monache, M.~Garavello, P.~Goatin, and B.~Piccoli,
  \emph{Control problems for conservation laws with traffic applications:
  modeling, analysis, and numerical methods}.\hskip 1em plus 0.5em minus
  0.4em\relax Springer Nature, 2022.

\end{thebibliography}

\end{document}